\documentclass[11pt,a4paper]{amsart}
\usepackage{mathptmx} % rm & math
\usepackage[scaled=0.90]{helvet} % ss
\usepackage{courier}
\normalfont
\usepackage[T1]{fontenc}
\usepackage{amssymb}

\renewcommand{\epsilon}{\varepsilon}
\renewcommand{\setminus}{\smallsetminus}
\renewcommand{\emptyset}{\varnothing}

\newtheorem{theorem}{Theorem}[section]
\newtheorem{proposition}[theorem]{Proposition}
\newtheorem{corollary}[theorem]{Corollary}
\newtheorem{lemma}[theorem]{Lemma}

\theoremstyle{definition}
\newtheorem{example}[theorem]{Example}

\newtheorem{definition}[theorem]{Definition}

\theoremstyle{remark}
\newtheorem{remark}[theorem]{Remark}

\newcommand{\Z}{\mathbb Z}

\newcommand{\R}{\mathbb R}

\newcommand{\BR}{\mathbb{R}}

\newcommand{\BZ}{\mathbb{Z}}

\newcommand{\FP}{\operatorname{FP}}

\newcommand{\UFP}{\underline{\operatorname{FP}}}

\newcommand{\cohom}[3]{H^{{\raise1pt\hbox{$\scriptstyle#1$}}}(#2\>\!,#3)}
\newcommand{\tatecohom}[3]%
  {\widehat H^{{\raise1pt\hbox{$\scriptstyle#1$}}}(#2\>\!,#3)}

\newcommand{\Cohom}[3]%
  {H^{{\raise1pt\hbox{$\scriptstyle#1$}}}\big(#2\>\!,#3\big)}
\newcommand{\Tatecohom}[3]%
  {\widehat H^{{\raise1pt\hbox{$\scriptstyle#1$}}}\big(#2\>\!,#3\big)}

\newcommand{\homol}[3]{H_{{\lower1pt\hbox{$\scriptstyle#1$}}}(#2\>\!,#3)}
\newcommand{\homolog}[2]{H_{{\lower1pt\hbox{$\scriptstyle#1$}}}(#2)}

\newcommand{\epi}{\twoheadrightarrow}

\newcommand{\OFG}{\mathcal O_{\mathcal F}G}

\newcommand{\OFH}{\mathcal O_{\mathcal F}H}
\newcommand{\OFM}{\mathcal O_{\mathcal F}M}

\title{Sigma theory for Bredon modules}
\author{D.~ H. ~Kochloukova}

\address{Dessislava H.~Kochloukova, Department of Mathematics, University of Campinas, Cx. P. 6065,
13083-970 Campinas, SP, Brazil}\email{desi@unicamp.br}

\author{C.~Mart\'inez-P\'erez}

\address{Conchita Mart\'inez-P\'erez, Departamento de Matem\'aticas, Universidad de Zaragoza,
50009 Zaragoza, Spain} \email{conmar@unizar.es}

\date{\today} % Activate to display a given date or no date
\keywords{}
\subjclass[2000]{%\color{red}
20J05}

\thanks{}

\begin{document}

\thispagestyle{empty}

\begin{abstract} We develop new invariants  $\underline{\Sigma}^m(G, \underline{A})$ similar to the Bieri-Strebel-Neumann-Renz invariants $\Sigma^m(G, A)$ but in the category of Bredon modules $\underline{A}$ (with respect to the class of the finite subgroups of $G$). We prove that for virtually soluble groups of type $FP_{\infty}$ and finite extension of the Thompson group $F$ we have $
\underline{\Sigma}^{\infty}(G, \underline{\Z}) = \Sigma^{\infty}(G, \Z)$.
\end{abstract}

\maketitle

\section{Introduction}
Bredon cohomology with  respect to the family of finite subgroups can be intuitively understood as the cohomology theory obtained by considering proper (i.e., with finite stabilizers) actions instead of free actions of groups.
In this paper we introduce $\Sigma$-theory for the class of Bredon modules similar to the classical   Bieri-Strebel-Neumann-Renz  theory developed since 1980's. In the classical case modules $A$ over the group algebra $\Z G$ are considered and by definition a class $[\chi]$ of a non-trivial character $\chi : G \to \R$ is in $\Sigma^m(G, A)$ if $A$ is of type $FP_m$ over $\Z G_{\chi}$, where $G_{\chi}$ is the monoid  $\{ g \in G | \chi(g) \geq 0 \}$.  An early version of $\Sigma^1(G, \Z)$ was used as an important tool in the classification of all finitely presented metabelian groups by Bieri and Strebel \cite{BS2}. The importance of the invariant $\Sigma^m(G, \Z)$  lies in the fact that it classifies which subgroups of $G$ above the commutator are of type $FP_m$ \cite{BieriRenz}.
One of our main results, Theorem D below, is that the analogous statement holds for Bredon cohomology for the newly defined Bredon  $\Sigma$-invariants.

Finiteness cohomological conditions in Bredon cohomology play the same role when studying proper classifying spaces that ordinary finiteness cohomological properties for ordinary classifying spaces. Recall that for a group $G$ a $G$-$CW$ complex $X$ is a model for $\underline{\text{E}}G$, the proper classifying space, if $X^H$ is contractible whenever $H\leq G$ is finite and empty otherwise.

Then, if there is a model for $\underline{\text{E}}G$ with cocompact 2-skeleton and $G$ is of type Bredon $\FP_n$, also denoted $\underline{\FP}_n$ (see Section \ref{first} for a definition), one can show that there exists a model for $\underline{\text{E}}G$ with finite $n$-skeleton.
This follows using the $n$-dimensional version of \cite{lueck} Theorem 4.2, which in turn can be proven truncating at dimension $n$ the inductive procedure used there. In this paper we create a Bredon version of the homological $\Sigma$-invariants and hope that in future the question of homotopic Bredon $\Sigma$-invariants can be addressed.

We first develop general Bredon theory for modules over cancelation monoids
and later concentrate on monoids $G_{\chi}$, where $G$ is a group  and $\chi$ a non-zero real character of $G$. The main obstacle to develop a general $\Sigma$-theory is that the sets $[M/K,M/H]$ (see Section \ref{first} for notation), where $K$ and $H$ are finite subgroups of the cancelation monoid $M$, are not always finitely generated over the Weyl monoid $W_MK$ and to avoid this problem we consider special monoids $M$, namely cancelation monoids
 that conjugate finite subgroups. For these monoids we describe the Bredon type $\underline{FP}_n$ in the following result (see Corollary \ref{FPmproperty}).

\medskip
{\bf Theorem A.}
{\it A  cancelation monoid $M$
which conjugates finite subgroups is
of type $\UFP_n$ if and only if there are
finitely many finite subgroups $H_1,...,H_s$ such
that for each finite subgroup $K$ of $ M$ there is an
element $m \in M$ such that $Km \subseteq mH_i$
for some $i =1,...,s$ and $W_MK$ is of type
$\FP_n.$}

\medskip
In section \ref{sectionSigmaBredon} we define the new invariant
$\underline\Sigma^m(G, \underline{A})$ for an $\OFG$-module $\underline A$ and study it in detail for the trivial module $\underline{\Z}$. The following result (see  Theorem \ref{condition2}) classifies the elements of the new invariant in terms of the classical $\Sigma$-invariant.

\medskip

{\bf Theorem B.} {\it Suppose that $G$ is a finitely generated group and has finitely many conjugacy classes of finite subgroups.
Then
$[\chi] \in \underline{\Sigma}^m(G, \underline{\Z} )$ if and only if  there is a subgroup  $\widetilde G$ of finite index in $G$ that contains the commutator subgroup $G'$ and such that for every finite subgroup $K$ of $G$ we have $K\leq\widetilde G$ and

1. $N_{\widetilde{G}}(K) ( ker (\chi) \cap \widetilde{G}) = \widetilde{G}$;

2. $  \chi(N_G(K)) \not= 0 \hbox{ and }[\chi |_{N_G(K)}] \in \Sigma^m(N_G(K), \mathbb{Z}).$

Furthermore condition 2 can be substituted for condition

2b). $\chi(C_G(K)) \not= 0$ and $[\chi |_{C_G(K)}] \in \Sigma^m(C_G(K), \mathbb{Z})$.}

\medskip
By one of the main results in \cite{BieriRenz} the classical $\Sigma$ invariant is always an open subset of the character sphere $S(G)$. In the Bredon case the situation is slightly different and we have the following result (see Theorem \ref{open}).

\medskip

{\bf Theorem C.} {\it Suppose that $\underline{\Sigma}^m(G, \underline{\Z}) \not= \emptyset$. Then $\underline{\Sigma}^m(G, \underline{\Z})$ is open in $S(G)$ if and only if $N_G(K) G'$ has finite index in $G$ for every finite subgroup $K$.}

\medskip
In \cite[Thm.~B]{BieriRenz} it was shown that for a group $G$ of homological  type $FP_m$ and a subgroup $H$ of $G$ that contains the commutator of $G$ we have that $H$ is of type $FP_m$ if and only if $S(G, H) = \{ [\chi] |  \ \chi(H) = 0 \} \subseteq \Sigma^m(G, \mathbb{Z})$.
We establish the following Bredon version of \cite[Thm.~B]{BieriRenz} ( see Theorem \ref{BieriRenz}).

\medskip
{\bf Theorem D.}
{\it Let $H$ be a subgroup of $G$ that contains the commutator and $G/ H$ is torsion-free and non-trivial.
Then $\underline{\Z}$ is Bredon $FP_m$ as $\OFH$-module if and only if
 $
S(G, H) \subseteq  \underline{\Sigma}^m(G, \underline{\Z}).$}

\medskip

\medskip
Finally in the last two sections of the paper we consider the case of virtually soluble groups of type $FP_{\infty}$ or finite extension of the R. Thompson group $F$. In both cases the groups are known to be  of type Bredon $FP_{\infty}$ \cite{KMN1}, \cite{ConcBri}. The proofs of both cases of Theorem E (see Theorem \ref{Bredonsoluble} and Theorem \ref{BredonThompson})  use Theorem B and the  techniques developed to prove that $G$ is of type Bredon $FP_{\infty}$ in  \cite{KMN1}, \cite{ConcBri}.

\medskip

{\bf Theorem E.} {\it If $G$ is virtually soluble of type $FP_{\infty}$ or is a finite extension of the R. Thompson group $F$ then $
\underline{\Sigma}^{\infty}(G, \underline{\Z}) = \Sigma^{\infty}(G, \Z)$.}

\medskip

Acknowledgements : We thank Brita Nucinkis for the  fruitful talks on the new Bredon Sigma invariants which finally led to this article. The second author thanks the department of mathematics at UNICAMP, Brazil for the hospitality during a
visit of a week in September, 2011 when some parts of the current paper were developed and thanks FAPESP, Brazil for the travel grant for a visit to Brazil including the participation at a workshop on group theory in Ubatuba, September 2011. She was also supported by  Gobierno de Aragon, European Regional Development Funds and
MTM2010-19938-C03-03.  The first author  is partially supported by "bolsa de produtividade em pesquisa", CNPq, Brazil.

\section{Some Bredon cohomology for monoids}\label{first}

Let $M$ be a monoid.  We say  that $M$ is of type
$\FP_n$ if the trivial left module $\Z$ is of
type $\FP_n$ over the monoid ring $\Z M$. If not otherwise stated the modules considered in the paper are left ones.
Observe that for a monoid $M$ defining  type $FP_n$ using the right trivial $\Z M$-module $\Z$ might yield different result. Even for the monoid $M = G_{\chi} = \{ g \in G | \chi(g) \geq 0 \}$ that comes from a non-zero real character $\chi : G \to \R$, where $G$ is a finitely generated group,  we might have that the trivial left $\Z M$-module $\Z$ is $FP_m$ but   the trivial right $\Z M$-module $\Z$ is not $FP_m$.  For example if we consider the Bieri-Strebel-Neumann-Renz invariant $\Sigma^m(G, \Z)$ defined for left $\Z G$-modules $\Z$ it suffices that $[\chi] \in \Sigma^m(G, \Z)$ but $[- \chi] \not\in \Sigma^m(G, \Z)$.
 This is a consequence of the fact that if $\Z$ were $\FP_m$ as a right $M$-module, then via $g\mapsto g^{-1}$ one could show that $\Z$ would be $\FP_m$ as a left  $G_{-\chi}$-module, contradicting that  $[- \chi] \not\in \Sigma^m(G, \Z)$ .
 As shown in \cite{GuPr} even the notion of finite cohomological dimension for a monoid depends on the choice of left or right modules.

\begin{definition}\label{cancelation} We say that a monoid $M$ is a left (resp. right) cancelation monoid if for any $m, m_1, m_2 \in M$ such that $m m_1 = m m_2$ (resp. such that $m_1 m=m_2 m$) we have $m_1 = m_2$. And we say that $M$ is a cancelation monoid if it is both left and right cancelation monoid. \end{definition}

From now on until the end of this section we assume that $M$ is a  {\bf cancelation}  monoid unless otherwise stated.
We say that a (left) $M$-set $\Delta$ is
transitive if it is generated by a single $\omega\in\Delta$, i.e., if
 $\Delta=M\omega$. We say that an $M$-set $\Delta$ is restricted if
it is a disjoint union of transitive $M$-sets,
i.e.,
$$\Delta=\dot{\bigcup}_{\lambda\in\Lambda}\Delta_\lambda$$
where $\Lambda$ is a set and each
$\Delta_\lambda$ is transitive.

 If $\Lambda$ is
finite, we say that $\Delta$ is $M$-finite. And  if
 the stabilizer in $M$ of any generator of each $\Delta_\lambda$ is a
{\bf finite } submonoid, then we say that $\Delta$ is proper. Note that since $M$ is a cancelation monoid any finite submonoid is a subgroup.
For example, for the monoid $M=\{x^n\mid n\geq
0\}$ the $M$-set $\{x^n\mid n\in\Z\}$ with left $M$-action given by multiplication is not
restricted. Neither is the finitely generated $M$-set $X:=M\dot\cup M/\sim$, where $\sim$ consists of identifying the two copies of $x^i$ whenever $i\geq i_0$ for a fixed $i_0>0$.

We define the orbit category $\OFM$ to be the
category with objects the transitive proper $M$-sets.
We denote the objects of $\OFM$ by
$M/K$, where $K$ is a finite subgroup of $M$. Here $M / K = \{ m K \mid m \in M \}$.
Morphisms in $\OFM$ are $M$-maps $\phi: M/K \to
M/H$ and are uniquely determined by $\phi(K) =mH$.
For this to be well defined we need $Km
\subseteq mH.$ The set of morphisms
$mor(M/K,M/H)$ is denoted
  $$[M/K,M/H]=\{mH\mid Km\subseteq mH\}.$$
In the particular case when $K=H$ we set
$$W_MK:=[M/K,M/K].$$
 Note that since $M$ is a cancelation monoid the sets $m K $ and $Km$ have the same cardinality as $K$, hence $W_MK = \{ m K  \mid m K = K m \}$ and
$W_MK$ is a cancelation monoid which we call the
Weyl monoid for $K$ in analogy with Weyl groups.

As in the group case, we may define a
Bredon module, or $\OFM$-module $V(-)$, as a
contravariant functor from $\OFM$ to the category
of abelian groups. The Bredon modules form an abelian
category so we have (co)products and (co)limits,
exact sequences etc are defined analogously.

By definition a free Bredon module is one of the form
$\Z[-,\Delta],$ where $\Delta$ is a restricted proper
$M$-set.  We say that a morphism in the category of Bredon modules $V(-) \to W(-)$ is an epimorphism if  for every $M / K \in \OFM$ we have that the map $V(M/ K) \to W(M/K)$ is surjective. It is easy to see that every Bredon module is an epimorphic image of a free Bredon module. Following the same procedure as in the group
case, see \cite{mislinsurvey}, one can show that
the category of $\OFM$-modules has enough
projectives and then define cohomology and
homology.

Observe also that for any $\OFM$-module $V$ and
any  finite subgroup $K\leq M$,
the functor $V(-)$ yields a structure of $W_MK$-module in $V(M/K)$.
In the particular case when $V(-)=\Z[-,M/H]$ this action is given by $xKmH=xmH$ for $xK\in W_MK$, $m\in M$ (for example, if $M$ has no finite  subgroups, then a $\OFM$-module is just an $M$-module).
 But note that
the condition that $\Delta$ is restricted in the definition of a free Bredon module is
necessary, for example consider again the monoid
$M=\{x^n\mid n\geq 0\}$ and $\Delta=\{x^n\mid
n\in\Z\}$. Then the $\OFM$-module
$\Z[-,\Delta]$ is just the $M$-module $\Z\Delta$ which can not be projective (an easy way to see it is to observe that for any $a\in\Z\Delta$, there is some $b\in\Z\Delta$ with $xb=a$, something that can not happen in any submodule of a free $M$-module).

By definition an $\OFM$-set $\Sigma$ is a collection
of sets $\Sigma_K$, one for each finite subgroup $K\leq
M$. We say that $\Sigma$ is $\OFM$-finite if
$\Sigma_K$ is finite for each $K$  and empty for
all but finitely many subgroups $K$. As for
groups (\cite[Section 9]{lueckbook}), an
$\OFM$-module $U$ is finitely generated if there is an $\OFM$-finite $\OFM$-set $\Sigma$
such that for each finite subgroup $K$, $\Sigma_K\subseteq U(M/K)$ and there is no proper submodule $V(-)$ of $U(-)$ with $\Sigma_K\subseteq V(M/K)$ for any finite subgroup  $K$. If we put
$\Delta_K=M/K\times\Sigma_K$ seen as $M$-set with trivial action on the right hand factor and $\Delta:=\dot\bigcup\{\Delta_K\mid K\leq M\text{ finite subgroup}\},$
 then
 there is a surjection  $\rho : \Z[-,\Delta]
\epi U(-)$. And conversely, if there is such a surjection $\rho$ for some restricted $M$-finite proper $M$-set $\Delta$, then $U(-)$ is finitely generated.

The finiteness conditions $\UFP_n$,
$\UFP_\infty$ and $\UFP$ for $\OFM$-modules are
defined as usual and
we say a cancelation monoid $M$ satisfies any of the above
finiteness conditions if the trivial
$\OFM$-module $\underline\Z$ does, where $\underline\Z$ is the $\OFM$-module with $\underline\Z(M/K)=\Z$ for any $K$ finite and with all the maps equal to the identity.

In particular if $M=G$ is a group this defines the finiteness conditions  $\UFP_n$,
$\UFP_\infty$ and $\UFP$ for $\OFG$-modules, also called Bredon finiteness conditions. Our main objective in the rest of this section is to generalize to monoids (under reasonable extra hypotheses) the following  well known characterization

\begin{lemma}(\cite{lueck}, \cite[Lemma 3.2]{kmn})\label{ordinary} A group $G$ is of type $\UFP_n$ if and only if it has finitely many conjugacy classes of finite subgroups and moreover for each finite subgroup $K$, the Weyl group $W_GK=N_G(K)/K$ is of type $\FP_n$.
\end{lemma}

We begin with the case $n=0.$

\begin{lemma}\label{fp0monoids}
A cancelation monoid $M$ is of type $\UFP_0$ if and only if
there are finitely many finite subgroups
$H_1,...,H_n$ of $M$ such that for each finite subgroup
$K$ of $M$ there is an element $m \in M$ such that $Km
\subseteq mH_i$ for some $i =1,...,n.$
\end{lemma}

\proof  Note that $M$ is of type $\UFP_0$ if and
only if there is a restricted $M$-finite proper $M$-set
$\Delta=\bigcup_{i=1}^nM/H_i\times\Delta_{H_i}$ (i.e. every $\Delta_{H_i}$ is finite) such that
$\Z[-,\Delta]$ surjects onto $\underline{\Z}(-)$. This means
that for any  finite subgroup $K$, $\Z[M/K,\Delta]\neq 0$
thus there is an $m \in M$ such that $Km
\subseteq mH_i$ for some $i$.     \qed

\begin{definition} \label{conj} Let $M$ be a monoid. We say that $M$ conjugates finite subgroups if
for any $H_1,H_2\leq M$ finite such that
$H_1g=gH_2$ for some $g\in M$, there exists a
$h\in M$ invertible with $H_1^h=H_2$, where $H_1^h = h^{-1} H_1 h$. We say that
$M$ has finitely many conjugacy classes of finite
subgroups if there is a finite family of finite
subgroups $\{K_1,\ldots,K_s\}$ such that for any
finite subgroup  $L\leq M$  there is some $i$ and some
invertible $t\in M$ such that $L=t^{-1}K_it$.
\end{definition}

\begin{proposition}\label{conditionFPn}
Let $M$
be a cancelation monoid of type $\UFP_0$ that satisfies the following two conditions:
   \begin{itemize}
\item[i)]  For any  finite subgroups $K,L$ of  $ M$ the $W_MK$-module
$\Z[M/K,M/L]$ is of type $FP_{\infty}$.

\item[ii)] $M$ has finitely many conjugacy classes of finite subgroups.
\end{itemize}
Then a $\OFM$-module $V$ is of
type $\UFP_n$ if and only if for each finite  subgroup
$K$ of $M$,  $V(M/K)$ is a module of type $\FP_n$
over $W_MK.$
\end{proposition}

\proof  If $V$ is of type $\UFP_n$, then there is
a projective resolution $P_\ast (-) \epi V(-)$
such that $P_i(-)$ is finitely generated for all
$i \leq n.$ We may assume that for $i\leq n$ the
$P_i(-)$ are finitely generated free, hence of
the form $\Z[-,\Delta_i]$ with $\Delta_i$ a
restricted proper $M$-finite $M$-set. Thus
$\Z[-,\Delta_i]=\oplus_{j=1}^k\Z[-,M/L_j]$ for some
finite subgroups $L_j$ (here we allow repetitions). Note that for any
finite subgroup $K$, $P_\ast (M/K) \epi V(M/K)$
is an exact sequence of modules. Since each $P_i(M/K)$ is of type
$\FP_\infty$ as $W_MK$-module for $i \leq n$ this
yields  the result by dimension shifting
(\cite[Proposition 1.4]{Bieribook}).

\noindent Assume now that $V$ is an $\OFM$-module
such that each $V(M/K)$ is finitely generated as
$W_MK$-module. Choose a set $\Omega$ of
representatives of each conjugacy class of finite
subgroups, so $\Omega$ is finite. For each $K\in\Omega$,
let $\Sigma_K$ be a finite generating system of
$V(M/K)$ as $W_MK$-module. Then the $\OFM$-set
$\Sigma$ which corresponds to $\Sigma_K$ whenever
$K\in\Omega$ and is empty otherwise is
$\OFM$-finite and generates $V$ as $\OFM$-module.
As a consequence, there is a finitely generated
free $\OFM$-module $P$ and an epimorphism
$P\twoheadrightarrow V$.
This proves the case $n=0$ of the \lq\lq if"
part. For the general case, argue by induction,
exactly as in
  \cite[Lemma 3.1]{kmn}. Explicitly, assume that
the result holds  for $n-1$ and consider the
$n$-th kernel $U_n$ of a projective resolution of
$V$ as $\OFM$-module formed, up to degree $n-1$,
by finitely generated free Bredon modules. Then the fact
that evaluating a finitely generated free module
at each $M/K$ yields a $W_MK$-module of type
$\FP_\infty$ and the hypothesis that $V(M/K)$ is
$\FP_n$ imply that each $U_n(M/K)$ is finitely
generated and the preceding paragraph yields that $U_n$ is finitely generated, hence
$V$ is also $\UFP_n$.
   \qed

\begin{lemma}\label{restfpinfty} Let $M$ be a  cancelation monoid  and $\Delta$ be a restricted $M$-finite proper $M$-set. Then the permutation  $M$-module $\Z\Delta$ is of type $\FP_\infty$.
\end{lemma}
\begin{proof} Note that it suffices to consider the case when $\Delta$ is $M$-transitive i.e. we may assume that $\Delta=M/K$ for some finite  subgroup $K$ of $ M$.
Then the assertion is obvious as $\Z$
is of type $\FP_\infty$ as $\Z K$-module and the
induction functor from $K$-modules to $M$-modules
is exact and takes finitely generated projectives
to finitely generated projectives.
\end{proof}

\begin{lemma}\label{decomposition} Let $H,K\leq M$ be finite subgroups of the cancelation monoid $M$. There are subgroups $K_1,\ldots, K_s\leq H$ such that there is a  decomposition as $W_MK$-set
\begin{equation} \label{disjoint} [M/K,M/H]=\cup_{i=1}^s\Omega_i\text{ where }\Omega_i:=\{mH\mid Km=mK_i\}.\end{equation}
Furthermore any two sets $\Omega_i$ and $\Omega_j$ are either disjoint or equal. In particular if we take a decomposition where $s$  is minimal, the union in (\ref{disjoint}) is disjoint.
\end{lemma}
\begin{proof} Fix $mH\in[M/K,M/H]$ and let
$$\overline K:=\{h\in H\mid mh\in Km\}.$$
Given $h_1,h_2\in\overline K$, there are some
$k_1,k_2\in K$ such that $mh_1=k_1m$ and
$mh_2=k_2m$. Then $mh_1h_2=k_1mh_2=k_1k_2m$ thus
$\overline K$ is a subgroup of $H$. Obviously,
$m\overline K\subseteq Km$ and conversely, as
$Km\subseteq mH$ one gets $Km \subseteq m\overline K$, hence $Km=m\overline K$.
 As for $xK\in W_MK$,
$x m\bar K=xKm=Kx m$,
 the monoid $W_MK$
acts on $\{mH\mid
Km=m\overline K\}$. Since $H$ is finite, there are finitely many subgroups $\overline K$ that can be obtained in this form so the first assertion follows.

Suppose $\{mH\mid Km=mK_i\} \cap \{mH\mid Km=mK_j\} \not= \emptyset$ for some $i \not= j$, so there are
$a_1, a_2 \in M$ such that $a_1H = a_2 H$, $K a_1 = a_1 K_i$ and $K a_2 = a_2 K_j$. Then $a_1 = a_2 h$ for some $h \in H$ and hence $K a_2 h = K a_1 = a_1 K_i = a_2 h K_i$, so
$$
a_2 K_j = K a_2 = a_2 K_i^{h^{-1}}.
$$
Since  $M$ is a cancelation monoid $K_j = K_i^{h^{-1}}$, so
$\{mH\mid  Km=mK_j\} = \{ mH  \mid Kmh = mh K_i \} = \{m_0 H \mid  Km_0=m_0K_i\}$.
 \end{proof}

\begin{definition} We say that a monoid $M$ has the left linear property if for every $m_1, m_2 \in M$ at least one of the linear equations $x m_1 = m_2$ and $x m_2  = m_1$ has a solution in $M$.
\end{definition}

\begin{lemma} \label{transitivity1} Assume that $M$ is a cancelation monoid  with the left linear property. Let $\Omega_i$ be one of the disjoint sets from Lemma  \ref{decomposition}  and assume that $\Omega_i$ is finitely generated over $W_MK$. Then $\Omega_i$ is $W_MK$-transitive.
\end{lemma}

\begin{proof} Let $m_1H, m_2 H \in \Omega_i$. Since $M$ has the left linear property there is $f \in M$ such that $f m_1 = m_2$ or $f m_2 = m_1$. Assume $m_1 = f m_2$. Then $K f m_2 = K m_1 = m_1 K_i = f m_2 K_i = f K m_2$. Since $M$ is a cancelation monoid we deduce that $K f = f K$, so
$f K\in W_MK$. Hence $m_1 H \in W_MK (m_2 H)$.

Finally if $m_1H, \ldots, m_k H$ is a generating set of $\Omega_i$ over $W_MK$ then the previous paragraph implies that
$\Omega_i$ is transitive and reordering we may assume $\Omega_i = W_MK ( m_k H)$.

\end{proof}

\begin{lemma}\label{finite} Assume that $M$ is a cancelation monoid and that for
some fixed  finite subgroups $K,H$ of $M$  each of the $W_MK$-sets $\Omega_i$ of Lemma \ref{decomposition}
is $W_MK$-transitive. Then the $W_MK$-module
$\mathbb{Z} [M/K,M/H]$ is of type $\FP_\infty$.\end{lemma}
\begin{proof}
The hypothesis implies that the $W_MK$-set $[M/K,M/H]$ is $W_MK$-finite and restricted. By Lemma \ref{restfpinfty}, we only have to check that it is also proper.

Set $L_{i} = \{ x K \in W_MK | x m_{i} H = m_{i} H \}$, where $\Omega_i = W_MK (m_i H)$.
The fact that $M$ is a cancelation monoid implies that for $f_1, f_2 \in m_i H$ the linear equation $x f_1 = f_2$ has at most one solution in $M$, thus
$L_{i}$ is finite.
 Since it has the left cancelation property, it is a subgroup.
\end{proof}

\begin{lemma} \label{lemmanew0} Assume that $M$ is a type  $\UFP_0$  cancelation monoid
that conjugates finite subgroups. Then $M$ has
finitely many conjugacy classes of finite
subgroups.
\end{lemma}
\begin{proof} By Lemma \ref{fp0monoids} there are finitely
many finite subgroups $H_1,\ldots,H_s$ such that
for each finite subgroup $K$ of $ M$ we have $Km\subseteq mH_i$ for
some $m\in M$ and some $i=1,\ldots,s$. Then
following the first part of the proof of Lemma
\ref{decomposition} we see that $Km = m \bar{K}$, where $\bar{K}$ is a finite subgroup of $H_i$. Then since $M$ conjugates finite subgroups there is an invertible $t \in M$ such that $K^t = \bar{K} \leq H_i$.
\end{proof}

\begin{corollary}\label{conditionFP0} Let $M$
be a type  $\UFP_0$  cancelation  monoid
that conjugates finite subgroups.  Then a $\OFM$-module $V$ is of
type $\UFP_n$ if and only if for each finite
 subgroup $K$ of $M$,  $V(M/K)$ is a module of type $\FP_n$
over $W_MK.$
\end{corollary}

\proof  By Proposition \ref{conditionFPn}, Lemma \ref{finite} and Lemma \ref{lemmanew0}, it suffices to show that  for any finite  subgrousp $K,H$ of $M$ each of the the $W_MK$-sets $\Omega_i$ of Lemma \ref{decomposition} is transitive. To see it,  note that since $M$ conjugates finite subgroups  for each $K_i$ with $Km=mK_i$, we must have $K_i=K^{m_i}$ for some $m_i\in M$ invertible. Then $\{mH\mid Km=mK_i\}=W_KM(m_iH)$.
   \qed

\begin{corollary} \label{FPmproperty}
A  cancelation monoid $M$ which conjugates finite subgroups is
of type $\UFP_n$ if and only if there are
finitely many finite subgroups $H_1,...,H_s$ such
that for each finite subgroup $K$ of $M$ there is an
element $m \in M$ such that $Km \subseteq mH_i$
for some $i =1,...,s$ and the monoid $W_MK$ is of type
$\FP_n.$
\end{corollary}

\begin{proof} It follows from Lemma \ref{fp0monoids} and Corollary \ref{conditionFP0} applied for $V = \underline{\mathbb{Z}}$.
\end{proof}

\section{The case of monoids obtained from characters}

Let $G$ be a group and $\chi : G \to \R$ be a
non-zero homomorphism. Consider the monoid
$$G_{\chi} = \{ g \in G | \chi(g) \geq 0 \}.
$$
Note that every finite subgroup of $G$ is contained in
$G_{\chi}$.
Recall that
\begin{equation}  \label{remark} \hbox{ if }G_{\chi}\hbox{ is of type
}FP_m \hbox{ then
}G \hbox{ is of
type
}FP_m. \end{equation}
From now on fix the monoid $M = G_{\chi}$. Note
that $W_MK=(W_GK)_\chi = N_G(K)_{\chi} / K$.

\begin{lemma} Let $\chi : G \to \mathbb{R}$ be a non-zero character. Then $G_\chi$  is a cancelation monoid with the left linear  property.
\end{lemma}
\begin{proof} Let $f_1, f_2 \in G_{\chi}$. Then either $f_1^{-1} f_2 \in G_{\chi}$ or $f_2^{-1} f_1 \in G_{\chi}$, so $G_{\chi}$ has the left linear property. Since $G_{\chi}$ is embedable in a group it is a cancelation monoid.
\end{proof}

\begin{remark} As a consequence of this and of Lemma \ref{transitivity1}, we get that if $\chi : G \to \mathbb{R}$ is a non-zero character and $H,K$ are finite subgroups of $G$, then $[G_\chi/K,G_\chi/H]$ is finitely generated as $W_{G_\chi}K$-set if and only if each of the sets $\Omega_i$ in Lemma \ref{decomposition} is $W_{G_\chi}K$-transitive.
\end{remark}

\begin{lemma} \label{conjfinite} Let $\chi : G \to \mathbb{R}$ be a non-zero character. Then $G_\chi$
conjugates finite subgroups if and only if
$G_\chi\subseteq N_G(K)\text{Ker} (\chi) $ for any
finite subgroup $K$ of $ G.$ This is equivalent to $G = N_G(K)\text{Ker}(\chi)$.
\end{lemma}
\begin{proof} Assume first that $G_\chi$
conjugates finite subgroups and consider a finite subgroup  $K$ of $G$. Then, for any $m\in G_\chi$ if we put
$K_1:=K^m\leq G$, we have $mK_1=Km$ thus
$K_1=K^t$ for some $t\in G_\chi$ invertible, in
particular $t\in\text{Ker}(\chi)$. Since $mt^{-1}\in
N_G(K)$ we get $m \in N_G(K) t \subseteq N_G(K)\text{Ker}(\chi)$, so $G_{\chi} \subseteq N_G(K)\text{Ker}(\chi)$.

Conversely, assume $G_\chi\subseteq
N_G(K)\text{Ker} (\chi)$ and let $K,K_1\leq G_\chi$
 be finite subgroups such that $mK_1=Km$ for $m\in G_\chi$. Put $m=st$
with $s\in N_G(K)$ and $t\in\text{Ker}(\chi)$, then
$t\in G_\chi$ is invertible and $K^t=K^m=K_1$.

In addition note that $G = G_{\chi} \cup G_{\chi}^{-1}$ and $N_G(K)\text{Ker}(\chi)$ is a subgroup of $G$, so if  $G_\chi\subseteq N_G(K)\text{Ker}(\chi)$ we get $G = N_G(K)\text{Ker}(\chi)$.
\end{proof}

\begin{lemma}\label{conjfinite1}   Let $\chi : G \to \mathbb{R}$ be a non-zero character, $\widetilde G \leq G$ a subgroup and $N$ a normal subgoup of $G$ with $N\subseteq\text{Ker}(\chi)\cap \widetilde G$. Assume that
$\widetilde G_\chi\subseteq N_G(K)N$ for every finite subgroup $K$ of $\widetilde{G}$. Then $\widetilde G_\chi$  conjugates finite subgroups.
\end{lemma}
\begin{proof} As $$\widetilde G_\chi\subseteq N_G(K)N\cap\widetilde G\subseteq N_{\widetilde G}(K)N\subseteq N_{\widetilde G}(K)(\text{Ker}(\chi)\cap\widetilde G)$$
it suffices to use Lemma \ref{conjfinite}.
\end{proof}

\begin{lemma} \label{characterconjugates} Let $M = G_{\chi}$. Then  for all finite subgroups $K$ and $H$ of $M$,
$[M/K,M/H]$ is finitely generated as
$W_MK$-set if and only if  for every finite $K \leq M$ one of the following conditions holds:

 (i) $\chi(N_G(K)) = 0$ i.e. $N_G(K) \subseteq\text{Ker}(\chi)$;

 (ii) $\chi(N_G(K)) \simeq \mathbb{Z}$ i.e $N_G(K)\text{Ker}(\chi) /\text{Ker}(\chi)$ has torsion-free rank 1;

 (iii) $\chi(N_G(K)) = \chi(G)$ i.e. $N_G(K)\text{Ker}(\chi) = G$.
\end{lemma}

\begin{proof}  Let $K_1, \ldots , K_s$ be the subgroups in the decomposition of Lemma \ref{decomposition}. Note that these are precisely the subgroups of $H$ which are
 conjugated to $K$ by an element lying in $M$. Choose for each of them some $m_i\in M$ with $K^{m_i} = K_i$. Then $[M/K, M/H] =  \cup_{1 \leq i \leq s} \{ mH \mid K^{m} = K^{m_i} \}$.
Thus the set  $[M/K, M/H]$ is finitely generated over $W_MK$ if and only if for every $1\leq i\leq s$ the set $(N_G(K) m_i) \cap M$ is finitely generated over $N_M(K) = N_G(K) \cap M$  via left multiplication.
 This is equivalent to $\text{inf} \{ \chi (g m_i) |g \in N_G(K),  gm_i \in G_{\chi} \} $ being attained which is equivalent to $\text{inf} \{ \chi(g) | g \in N_G(K), \chi(g) \geq - \chi(m_i) \}$ being attained. Note that this is the case if one of the conditions (i), (ii) or (iii) holds.

Let $m_0 \in G_{\chi}$. Suppose that $[M/K,M/H]$ is finitely generated as
$W_MK$-set for $H = K^{m_0}$ and that neither of the conditions (i) and (ii) hold. Then the restriction of $\chi$ on $N_G(K)$ is a non-discrete non-zero real character, hence $\chi(N_G(K))$ is a dense subset of $\mathbb{R}$. Then
$[M/K, M/H] = \{ mH \mid Km \subseteq mH = m K^{m_0} \} = \{ m H \mid K^m \subseteq K^{m_0}\} =   \{ mH \mid K^{m} = K^{m_0} \}$ i.e. $s=1$ and $m_1 = m_0$ and
the above infimum is attained if and only if $  - \chi (m_0) \in \chi(N_G(K))$ i.e.
$\chi(G_{- \chi}) \subseteq \chi(N_G(K))$. Since $N_G(K)$ is a group this is equivalent to $\chi(G) = \chi(N_G(K))$ i.e. $G = N_G(K)\text{Ker}(\chi)$.
\end{proof}

\begin{corollary}  Let $N$ be the subgroup of $G$ that contains the commutator $G'$  and such that $N / G '$ is the torsion part of $G/G'$. Then the assumptions of Lemma \ref{characterconjugates} hold for every non-zero character $\chi : G \to \mathbb{R}$  if for every finite subgroup  $K$ of $ G$   one of the following conditions holds :

(i) $N_G(K) \subseteq N$;

(ii) $N_G(K) N / N \simeq \mathbb{Z}$;

(iii) $N_G(K) N = G$.
\end{corollary}

\begin{remark} For different $K$ we may need different conditions from the list above.
\end{remark}
\begin{proof} Take a character $\chi_0 : G \to \mathbb{R}$ with $\text{Ker}(\chi_0) = N$. Then if Lemma \ref{characterconjugates} holds for any character $\chi$ it holds for $\chi = \chi_0$ and we are done.

For the converse let $\chi : G \to \mathbb{R}$ be an arbitrary non-zero character.
Suppose  $N_G(K) \subseteq N$. Then $N_G(K) \subseteq N \subseteq \text{Ker}(\chi)$.
If $N_G(K) N / N \simeq \mathbb{Z}$ then since $\chi (N) = 0$ we get that $\chi(N_G(K)) $ is cyclic i.e. is either zero or $\chi(N_G(K)) \simeq \mathbb{Z}$.
Finally if  $N_G(K) N = G$, applying $\chi$ we get $\chi(N_G(K)) = \chi(G)$.
\end{proof}

\begin{lemma} \label{pascoa} Let $M = G_{\chi}$.  Then $M$ has finitely many conjugacy classes of finite subgroups
if and only if the following two conditions hold

(i) $N_G(K) \text{Ker}(\chi)$ has finite index in $G$ for every finite subgroup $K$;

(ii)  $G$ has finitely many conjugacy classes of finite subgroups.
\end{lemma}

\begin{proof} Suppose first that $M$ has finitely many conjugacy classes of finite subgroups (see Definition \ref{conj}). Then since any finite subgroup of $G$ lies in $M$ (ii) follows immediately. Let $K_1, \ldots , K_r$ be representatives of the conjugacy classes of finite subgroups in $G$.
 Then for each $i$, the $G$-orbit (via right conjugation) generated by $K_i$ is finitely generated over the subgroup of invertible elements of $G_\chi$, i.e. $\text{Ker}(\chi)$. This is equivalent to the existence of a finite set $m_{i,1}, \ldots , m_{i,j_i}$ such that  $G = \cup_{1 \leq j \leq j_i} N_G(K_i)\text{Ker}({\chi}) m_j$. The last is equivalent to  $N_G(K)\text{Ker}(\chi)$ having finite index in $G$.

The converse follows by repeating the above argument backwards.
\end{proof}

\section{Some general facts about finite index extensions}

Assume a group $T$ acts by conjugation on a finitely generated group  $G$. This induces a right action of $T$ on $Hom(G, \mathbb{R})$
given by:
$$\chi^t(g):= \chi(g^{t^{-1}}).$$
Moreover, this actions leaves $\Sigma^n(G,\Z)$ setwise fixed for any $n$.

\begin{lemma}\label{finext} Let $G$ be a finitely generated group with a finite index subgroup $H$.  Then:
\begin{itemize}
\item[i)] For $\nu_1,\nu_2\in Hom(G, \mathbb{R})$, $\nu_1=\nu_2$ if and only if $\nu_1 |_H= \nu_2 \mid_H$.

\item[ii)] If $G=K\ltimes H$ for some (finite) $K$, the characters of $H$ that can be extended to $G$ are precisely those in the fixed points  $Hom(H, \mathbb{R})^G$.

\end{itemize}
\end{lemma}
\begin{proof} For i), note that for any $g\in G$, there is some $n>0$ with $g^n\in H$. Then, if $\nu_1 |_H=\nu_2 |_H$,
$$n\nu_1(g)=\nu_1(g^n)=\nu_2(g^n)=n\nu_2(g).$$
And for ii) observe that $K\leq\text{Ker}(\chi)$ for any character $\chi:G\to\R$ thus $\chi |_H\in Hom(H, \R)^G$. Conversely, let $\nu\in Hom(H, \mathbb{R})^G$ and define $\chi:G\to\R$ by $\chi(kh):=\nu(h)$  for $k \in K, h \in H$. This is a well defined character of $G$ that extends $\nu$.
\end{proof}

\section{Examples}

We will see later on  that property (i) from Lemma \ref{pascoa} i.e.
\begin{equation}\label{property}
\text{
$|G: N_G(K) \text{Ker} (\chi)|<\infty$ for any character
$\chi:G\to\R$, for any finite $K\leq G$ }
\end{equation}
holds for all virtually soluble groups of type $FP_{\infty}$ and  finite extensions of the R. Thompson group $F$ (see Theorem \ref{finalThompson} and Theorem \ref{Bredonsoluble}). Note that this is equivalent to $|G: N_G(K) [G,G]|<\infty$.
We will discuss these examples in more details in the last two sections of the paper.

\begin{example} Using right angled Artin groups, it is not difficult to construct groups for which condition (\ref{property}) does not hold true.
Let $K=C_2$ and $L$ be the simplicial complex having four vertices labeled 1,2,3 and 4 and edges joining the vertices labeled $\{1,4\}$; $\{2,4\}$ and $\{3,4\}$. Let $A_L$ be the associated right angled Artin group, i.e., the group with presentation
$$\langle g_1,g_2,g_3,g_4\mid [g_1,g_4],[g_2,g_4],[g_3,g_4] \rangle$$
 and consider the action of $K$ on $L$ given by swapping the vertices labeled 1 and 2. This yields and action of $K$ on $A_L$ and we may form the semidirect product $G=K\ltimes A_L$. Let $L^K$ denote the fixed point subcomplex, i.e., the subcomplex of $L$ consisting of the vertices 3, 4 with an edge joining them.
Then using Theorem 2 in \cite{learynucinkis} (due to Crisp) we get $N_G(K)/K\cong C_{A_L}(K)=A_{L^K} = \langle g_3, g_4 | [g_3, g_4] \rangle$.

Now, consider the character $\chi:G\to\R$ such that $\chi(g_1)=1=\chi(g_2)$, $\chi(g_3)=\alpha$, $\chi(g_4)=\beta$ so that $\langle 1,\alpha,\beta\rangle$ is a rank 3 subgroup of $(\R , +)$. Then
$\chi(N_G(K)\text{Ker}(\chi)) = \chi(N_G(K)) = \langle \alpha, \beta \rangle$ and $\chi(G) = \langle 1, \alpha, \beta \rangle$, so
$|G:N_G(K)\text{Ker}\chi|$ is not finite.
\end{example}

\begin{example} Finally we consider an example where (\ref{property}) holds
but for some character $\chi$ we have $G \not= N_G(K)\text{Ker}(\chi)$.s
This is equivalent to $\chi(G) \not= \chi(N_G(K))$.

Consider $G=K\ltimes\langle b_0,b_1,b_2,b_3\rangle$
with $K=C_2$ generated by $t$ swapping $b_0,b_1$ and $b_2,b_3$, where  $\langle b_0,b_1,b_2,b_3\rangle \simeq \mathbb{Z}^4$.
Let $a_0=b_0b_1$, $a_1=b_2b_3$, $a_2=b_0b_1^{-1}$, $a_3=b_2b_3^{-1}$ and  note that $a_2, a_3 \in G'$.
Then  $N_G(K) = C_G(K) = K \times \langle a_0, a_1 \rangle$, hence $\chi(N_G(K)) = \langle \chi(a_0) = 2 \chi(b_0), \chi(a_1) = 2 \chi(b_2) \rangle \not= \langle
\chi(b_0), \chi(b_2)  \rangle = \chi (G)$.
\end{example}

\section{Bredon Sigma  theory} \label{sectionSigmaBredon}

\noindent
Let $$S(G):=Hom (G, \R) \setminus \{ 0 \} / \sim$$
where $\sim$ is the equivalence relation  given
by $\chi_1 \sim \chi_2$ if $\chi_1  \in \R_{\small{>0}}
\chi_2$.  Write $[\chi]$ for the equivalence class of $\chi$.

\begin{definition}
Let $A$ be an $\OFG$-module and $[\chi] \in
S(G)$. Then we say that
$[\chi]\in\underline\Sigma^m(G, A)$ if there is
a subgroup $\widetilde{G}$ of finite index in  $G$ that contains all finite subgroups of $G$, $G' \leq \widetilde{G}$,
$M=\widetilde G_\chi$ conjugates finite subgroups and
$A$ is $\UFP_m$ as $\OFM$-module. Observe that in this definition  $\widetilde{G}$ might depend on $\chi$.
\end{definition}

\begin{theorem} \label{condition} Let $G$ be a finitely generated group. Then
$[\chi] \in \underline{\Sigma}^m(G, \underline{\Z} )$ if and only if there is
a subgroup  $\widetilde G$ of finite index in $G$ with $G'\leq\widetilde G$  and a family $\{H_1, \ldots , H_s \}$ of finite subgroups of $\widetilde{G}$  such that for any finite subgroup $K$ of $ G$ we have $K\leq\widetilde G$ and the following three conditions hold :

1. $N_{\widetilde{G}}(K) ( Ker (\chi) \cap \widetilde{G}) = \widetilde{G}$;

2. $\chi(N_G(K)) \not= 0$ and $[\chi |_{N_G(K)}] \in \Sigma^m(N_G(K), \mathbb{Z})$;

3.  there is  an element $m \in \widetilde{G}_{\chi}$ such that $Km \subseteq mH_i$ for some $i =1,...,s$.
\end{theorem}

\begin{remark} \label{centr01}  Since $C_G(K)$ has finite index in $N_G(K)$ and by \cite[Thm.~9.3]{MMV} (see Theorem \ref{finiteindex}) condition 2 is equivalent to condition

\medskip
2b). $\chi(C_G(K)) \not= 0$ and $[\chi |_{C_G(K)}] \in \Sigma^m(C_G(K), \mathbb{Z})$.
\end{remark}
\begin{proof} Note that by Lemma \ref{conjfinite} condition 1 is equivalent to $M= \widetilde G_\chi$ conjugates finite subgroups.
By Lemma \ref{fp0monoids} condition 3 is equivalent to $\underline{\Z}$ is $\UFP_0$ as $\OFM$-module.

By Corollary \ref{FPmproperty} $\underline{\Z}$ is $\UFP_m$ as $\OFM$-module if and only if condition 3 holds and $\mathbb{Z}$ is $FP_m$ as left $\mathbb{Z} W_MK$-module for every finite subgroup $K$. This is equivalent to \begin{equation} \label{eq01} \mathbb{Z} \hbox{ is of type } FP_m \hbox{ as left }\mathbb{Z} N_M(K)\hbox{-module}. \end{equation} Note that $N_M(K) = N_{\widetilde{G}}(K) \cap G_{\chi}$. By condition 1 we have that the restriction of $\chi$ on $N_{\widetilde{G}}(K)$ is non-zero, hence (\ref{eq01}) is equivalent to
\begin{equation} \label{pascoa2} [\chi |_{N_{\widetilde{G}}(K)}] \in \Sigma^m(N_{\widetilde{G}}(K), \mathbb{Z}).\end{equation} Since $N_{\widetilde{G}}(K)$ has finite index in $N_G(K)$ by \cite[Thm.~9.3]{MMV} (\ref{pascoa2}) is equivalent to
$[\chi] \in \Sigma^m(N_G(K), \mathbb{Z})$.
\end{proof}

\begin{lemma} \label{nonempty} Let $G$ be a finitely generated group. If $\underline{\Sigma}^m(G, \underline{\Z}) \not= \emptyset$ then $\underline{\Z}$ is Bredon $FP_m$.
\end{lemma}

\begin{proof}
By Lemma \ref{ordinary} we have to show that there are finitely many $G$-orbits under conjugation of finite subgroups in $G$ and $N_G(K)$ is of
type $FP_m$ for every finite subgroup $K$ of $G$.
Let $[\chi] \in \underline{\Sigma}^m(G, \underline{\Z} )$. Then condition 3 from Theorem \ref{condition} shows that any finite subgroup of $G$ is inside $H_i^g$ for some $g \in G_{- \chi}$ and $1 \leq i \leq s$, hence there are finitely many $G$-orbits of finite subgroups in $G$.

If $N_G(K)_{\chi}$ is of type $FP_m$ then $N_G(K)$ is of
type $FP_m$. Hence $\underline{\Z}$ is Bredon $FP_m$.
\end{proof}

\begin{remark}\label{subset} Let $G$ be a finitely generated group. Then $\underline{\Sigma}^m(G, \underline{\Z})\subseteq\Sigma^m(G,\Z)$. To see it, observe that we may assume $\underline{\Sigma}^m(G, \underline{\Z})\neq\emptyset$ and it suffices to consider the case when $K$ is the trivial group in part 2 from Theorem \ref{condition}.
\end{remark}

\begin{theorem} \label{condition2} Suppose that $G$ is a finitely generated group and has finitely many conjugacy classes of finite subgroups.
Then
$[\chi] \in \underline{\Sigma}^m(G, \underline{\Z} )$ if and only if  there is a subgroup  $\widetilde G$ of finite index in $G$ that contains the commutator subgroup $G'$ such that for every finite subgroup $K$ of $G$ we have
$K\leq\widetilde G$ and

1. $N_{\widetilde{G}}(K) ( ker (\chi) \cap \widetilde{G}) = \widetilde{G}$;

2. $  \chi(N_G(K)) \not= 0 \hbox{ and }[\chi |_{N_G(K)}] \in \Sigma^m(N_G(K), \mathbb{Z}).$
\end{theorem}

\begin{remark} As before condition 2 can be substituted with condition 2b) from  Remark \ref{centr01}.  By Lemma \ref{nonempty} if $\underline{\Sigma}^m(G, \underline{\Z}) \not= \emptyset$ then $G$ has finitely many conjugacy classes of finite subgroups.
\end{remark}

\begin{proof} Assume that there is $\widetilde{G}$ such that conditions 1 and 2 hold.
Since $G$ has finitely many conjugacy classes of finite subgroups, the same holds for $\widetilde{G}$.
Let $H_1, \ldots, H_s$ be representatives of the finitely many conjugacy classes of finite subgroups in $\widetilde{G}$.
 We claim that  condition 1 from  Theorem \ref{condition} implies condition 3 from Theorem \ref{condition}. Indeed condition 3 is equivalent to every
 finite subgroup $K$ of $G$ being a subgroup of $H_i^g$ for some $g \in \widetilde{G}_{- \chi}$ and $ 1 \leq i \leq s$.  Note that $K \leq H_i^t$ for some $t \in \widetilde{G} = N_{\widetilde{G}}(H_i) (ker (\chi) \cap \widetilde{G})$, so $t = xy$, $x \in  N_{\widetilde{G}}(H_i)$, $y \in ker (\chi) \cap \widetilde{G} \subseteq \widetilde{G}_{- \chi}$ and hence $H_i^t = H_i ^{xy} = H_i^y$.
\end{proof}

The following is a Bredon version of \cite[Thm.~B]{BieriRenz}.
 Recall that $S(G,H)$ consists of the classes of all those characters vanishing on $H$.
\begin{theorem}\label{BieriRenz}
Let $H$ be a subgroup of  a finitely generated group $G$ that contains the commutator subgroup $G'$ and such that $G/ H$ is torsion-free and non-trivial. Assume that  $\underline{\Z}$ is Bredon $FP_m$ as $\OFG$-module.
Then $\underline{\Z}$ is Bredon $FP_m$ as $\OFH$-module if and only if
 $
S(G, H) \subseteq  \underline{\Sigma}^m(G, \underline{\Z}).$
\end{theorem}

\begin{proof}
1. Suppose that $\underline{\Z}$ is Bredon $FP_m$ as $\OFH$-module. Then
 by Lemma \ref{ordinary}

a1) there are finitely many conjugacy classes of finite subgroups in $H$;

b1) for any finite subgroup $K$ of $H$ the group $N_H(K)$ is of type $FP_m$.

 Since $G,H$ are $\underline{\FP}_m$ there are finitely many conjugacy classes of finite subgroups in both $G,H$ and every $G$-orbit (of finite groups under conjugation) splits into finitely many $H$-orbits, hence for every finite subgroup $K$ of $G$
\begin{equation} \label{tildaG}
N_G(K) H \hbox{ has finite index in } G.
\end{equation}
Define
\begin{equation}
\widetilde{G} = \cap N_G(K) H,
\end{equation}
where the intersection is over representatives of the $G$-orbits of conjugacy classes of finite subgroups in $G$. Thus the intersection is finite, $\widetilde{G}$ has finite index in $G$ and contains $H$, hence $\widetilde{G}$ contains all the finite subgroups of $G$ and the commutator subgroup $G'$.

Let $[\chi] \in S(G, H)$ i.e. $H \subseteq \text{Ker}(\chi)$.
By Lemma \ref{conjfinite1}  applied for $N = H$, $\widetilde{G}_{\chi}$ conjugates finite subgroups.
Then
 condition 1 from Theorem \ref{condition2} holds and $\widetilde{G}$ is global i.e. $\widetilde{G}$ does not depend on $\chi$.

Note that $N_H(K) = N_{G}(K) \cap H$, so $N_G(K) / N_H(K)$ is abelian. By condition b1) and
\cite[Thm.~B]{BieriRenz} we have that every non-zero real character $\widetilde{\chi} : N_G(K) \to \mathbb{R}$ such that $\widetilde{\chi}(N_H(K)) = 0$ represents an element of $\Sigma^m(N_G(K), \mathbb{Z})$. In particular this holds for $\widetilde{\chi}$ the restriction of $\chi$ to $N_G(K)$, hence condition 2 from  Theorem \ref{condition2} holds. Note that since $\chi(H) = 0$ and $[G : N_G(K)H] < \infty$ we have $\chi(N_G(K)) \not= 0$. Then by  Theorem \ref{condition2}
$[\chi] \in \underline{\Sigma}^m(G, \underline{\Z})
$

2. Conversely suppose that
$$
S(G, H) \subseteq  \underline{\Sigma}^m(G, \underline{\Z}).$$
 Then by Theorem \ref{condition}  for every $[\chi] \in S(G, H)$ :

a2) there is a subgroup $\widetilde G$ of finite index in $G$ that contains the commutator $G'$ and
contains every finite subgroup $K$ of $G$  and $N_{\widetilde{G}}(K) ( Ker (\chi) \cap \widetilde{G}) = \widetilde{G}$;

b2) for any finite subgroup $K$ of $G$ we have $\chi(N_G(K)) \not= 0$  and $[\chi |_{N_G(K)}] \in \Sigma^m(N_G(K), \mathbb{Z})$. In particular $N_G(K)$ is of type $FP_m$.

Since $\underline{\Sigma}^m(G, \underline{\Z}) \not= \emptyset$, by Lemma \ref{nonempty}
 $G$ has finitely many conjugacy classes of finite subgroups. By a2) $N_G(K)\text{Ker}(\chi)$ has finite index in $G$, hence  every $G$-orbit  of finite subgroups ( i.e. $\{ K^g \}_{g \in G}$) splits into finitely many $\text{Ker}(\chi)$-orbits. Choose $\chi$ such that $H =\text{Ker}(\chi)$. Then every $G$-orbit of finite subgroups  splits into finitely many $H$-orbits, so $H$ has finitely many conjugacy classes of finite subgroups i.e. condition a1) holds.

It remains to show that condition b1) holds. Fix one finite subgroup $K$ of $G$. By b2) $N_G(K)$ is of type $FP_m$. Recall that $N_G(K) / N_H(K)$ is abelian.
Let $\mu : N_G(K) \to \mathbb{R}$ be a non-zero real
character such that $\mu(N_H(K)) = 0$. Then $\mu$ can be
extended to a real character $\mu_1$ of $N_G(K) H$ that is zero on
$H$. Since $G' \subseteq H$ we see that $\mu_1$ is
extendable to a real character $\chi$ of $G$. Thus by condition b2)
$N_G(K)_{\chi} = N_G(K)_{\mu}$ is of type $FP_{m}$. Then by the original Bieri-Renz
criterion \cite[Thm. B]{BieriRenz}  $N_H(K)$ is of type $FP_m$.
\end{proof}

\begin{example}
Let  $G$ be  a polycyclic group. Then $\underline{\Z}$ is
Bredon $FP_{\infty}$ as $\OFG$-module.
Let $H$ be the subgroup of $G$ that contains the
commutator and such that $H/G'$ is the torsion part of $G/G'$. Since $H$ is polycyclic,  $\underline{\Z}$ is  Bredon $FP_{\infty}$ as $\OFH$-module. So by the previous theorem
$$
\underline{\Sigma}^m(G, \underline{\Z}) = S(G, H) =  S(G).
$$
\end{example}

\begin{example}  Leary-Nucinkis have constructed an example of a group which is of ordinary type $\FP_\infty$ but not Bredon $\FP_\infty$ (see \cite{learynucinkis}). Their example is obtained as a finite index extension of a Bestvina-Brady group $B_L$, where $L$ is certain flag complex on which the alternating group of degree 5, $K=A_5$ acts. Consider the associated right angled Artin group $A_L$, then $K$ also acts on $L$ and we may form $G=K\ltimes A_L$. The map $\chi:A_L\to\Z$ which sends all the generators in the standard right-angled Artin presentation of $A_L$ to the identity is a (discrete) character of $G$ which by Lemma \ref{finext} can be lifted to a character $\chi_0$ of $G$, then $\text{Ker} (\chi) =B_L$ and $\text{Ker}(\chi_0) =K\ltimes B_L$ is Leary-Nucinkis' example.
Then Theorem \ref{BieriRenz} implies that $[\chi_0]\not\in\underline\Sigma^\infty(G,\underline{\Z})$, however $[\chi] \in\Sigma^\infty(G,\Z)$ by the version of this result for ordinary characters, see \cite[Thm.~B]{BieriRenz}. Note that by \cite{learynucinkis}, $L^K=\emptyset$ which implies that $K=N_G(K)$, so here  both conditions from Theorem \ref{condition2} fail.
\end{example}

\begin{theorem} \label{open} Suppose that $\underline{\Sigma}^m(G, \underline{\Z}) \not= \emptyset$. Then $\underline{\Sigma}^m(G, \underline{\Z})$ is open in $S(G)$ if and only if $N_G(K) G'$ has finite index in $G$ for every finite subgroup $K$.
\end{theorem}

\begin{remark} It will follow from the results in the last two sections of this paper  that the condition $[G : N_G(K) G'] < \infty$ holds for virtually soluble groups of type $FP_{\infty}$ and for finite extensions of the Thompson group $F$. \end{remark}

\begin{proof} Let  $N / G'$ be the torsion part of the abelianization $G / G'$.
By Lemma \ref{nonempty} $G$ has finitely many conjugacy classes of finite subgroups.

Suppose first that $\underline{\Sigma}^m(G, \underline{\Z})$ is open in $S(G)$ and let $[\chi] \in \underline{\Sigma}^m(G, \underline{\Z})$. Then there is $[\chi_0] \in \underline{\Sigma}^m(G, \underline{\Z})$ that is "close"' to $[\chi]$ and such that $\text{Ker}(\chi_0) = N$. Then condition 1 from Theorem \ref{condition2} implies that there is a subgroup of finite index $\widetilde{G}$ in $G$ such that $N_{\widetilde{G}}(K) (\text{Ker}(\chi_0) \cap \widetilde{G}) = \widetilde{G}$, hence $N_G(K) N$ has finite index in $G$ and since $N/G'$ is finite $N_G(K) G'$ has finite index in $G$.

For the converse assume that  $N_G(K) G'$ has finite index in $G$ for every finite subgroup $K$. Then $[G : N_G(K) N] < \infty$.  Let $\widetilde{G} = \cap_K N_G(K) N$ where the intersection is over representatives of conjugacy classes of finite subgroups in $G$, thus the intersection is finite and $\widetilde{G}$ has finite index in $G$. Note that $\widetilde{G} = N_{\widetilde{G}}(K) N$. Then by Lemma \ref{conjfinite} and Lemma \ref{conjfinite1}  condition 1 from Theorem \ref{condition2} holds for every character $\chi$ and furthermore $\chi (N_G(K)) \not= 0$.
Observe that since $N_G(K) N$ has finite index in $G$  we have $S(G) \subseteq S(N_G(K) N)$ and $S(G) = S(N_G(K), N \cap N_G(K)) \subseteq S(N_G(K))$.
Since $\Sigma^m(N_G(K), \mathbb{Z})$ is an open subset of $S(N_G(K))$ we deduce that condition 2 from Theorem \ref{condition2} is an open condition for a fixed finite subgroup $K$ i.e. if it holds for one character holds for a neighbourhood. The fact that there are finitely many conjugacy classes $K$ of finite subgroups in $G$  completes the proof.
\end{proof}

\section{Finite extensions of the Thompson group}

\noindent In this section we use  the notation of \cite{BrinGuzman}, \cite{KMN1}. So we denote by  $F_{n,
\infty}$
  the group of PL increasing homeomorphisms $f$ of $\BR$ acting on the
right
such that the set $X_f$ of break points of $f$ is a
discrete subset of $\BZ[{1 \over n}]$,
$f(X_f) \subseteq
\BZ[\frac{1}{n}]$ and slopes are integral powers of $n$. Furthermore,
there
are integers $i$ and
$j$
(depending on $f$) with
$$(x)f = \begin{cases} x + i (n-1)  \mbox{ for } x > M,  \\
        x + j (n-1) \mbox{ for }  x < - M \end{cases}$$
  for sufficiently large $M$ (depending on $f$ again).

\begin{definition} Let $n \geq 2$ and let $t_0\in\R$. Set
$$
F_{n,[t_0,\infty]} = \{ f \mid  f \hbox{ is the restriction to }[t_0,\infty]  \hbox{ of }
\tilde{f} \in F_{n, \infty}, \tilde{f}(t_0) =t_0\},$$
$$ F_{n, [-\infty, t_0]}= \{ f \mid  f \hbox{ is the restriction to }[-\infty,t_0]  \hbox{ of }
\tilde{f} \in F_{n, \infty}, \tilde{f}(t_0) =t_0\}.$$
Let $$\mu_1,\mu_2 :  F_{n, [t_0, \infty]} \to \mathbb{R} $$ be the characters given by $\mu_1(f) = log_n ((t_0)f')$ and $\mu_2(f)
= - i$ if
 $(x)f = x + i (n-1)$ for $x>>0$ and let
$$\nu_1,\nu_2 :  F_{n, [-\infty, t_0]} \to \mathbb{R} $$ be the characters given by  $\nu_1(f) = log_n ((t_0)f')$, $\nu_2(f) =  - j$ where  $(x)f = x + j (n-1)$ for $x<<0$.\end{definition}

The following lemma is a correction of  Lemma 4.7 from \cite{KMN1}. Note that in \cite{KMN1} though not explicitly stated the $\Sigma$-invariants are defined via right actions as both Proposition 3.3 and Theorem 3.4 there work only in this case.

However in this paper all modules are left ones, so we stick to $\Sigma$-invariants defined via left actions. Generally moving to a definition of the $\Sigma$-invariants from right actions to left actions changes  only the sign of the invariant i.e. we get the antipodal set.

\begin{lemma} \label{correction}
$\Sigma^1(F_{n,[t_0,\infty]})^c=\{ -[\mu_1], -[\mu_2]\} \hbox{ and }\Sigma^1(F_{n,[-\infty,t_0]})^c=\{ -[\nu_1],[\nu_2]\}.$
\end{lemma}

\begin{proof}
As shown in Lemma 4.7 in \cite{KMN1}, the map $\mu:[0,\infty)\to[0,n-1]$ of \cite{BrinGuzman} Lemma 2.3.1 induces by conjugation an isomorphism from $F_{n,[t_0,\infty]}$ to a subgroup of the group of piecewiselinear homeomorphisms of the interval $[t_0,n-1]$. Note that we can assume that $t_0$ is an element of any fixed interval $[s, s+ n-1)$ or $(s, s + n-1]$ since for $\delta : x \to x+ n-1$ we have $ \delta F_{n, [t_0, \infty]} \delta^{-1} = F_{n, [t_0- (n-1), \infty]}$. In particular we assume that $t_0 \in [0, n-1)$.
The map $\mu$ was used in Lemma 4.7 from \cite{KMN1} together with the description of $\Sigma^1$ for groups of PL automorphisms of closed intervals \cite{BNS}  to calculate $\Sigma^1(F_{n,[t_0,\infty]})^c$. Adding the extra  minus signs explained  in the paragraph before Lemma \ref{correction}  we get that   $\Sigma^1(F_{n,[t_0,\infty]})^c=\{ - [\mu_1], - [\mu_2]\}$.

Consider now the group $F_{n,[-\infty,t_0]}$. If we want to do the same as for the group $F_{n,[t_0,\infty]}$, we first have to modify $\mu$ by composing it with the map $x\mapsto-x$ so that we get $\mu^-:(-\infty,0]\to[0,n-1]$. Here we need to assume $-n+1 < t_0\leq 0$, exactly as for the argument before we needed $0\leq t_0 < n-1$.  This has the effect that now we get
$\Sigma^1(F_{n,[-\infty,t_0]})^c=\{ - [\nu_1],[\nu_2]\}$. Thus  the correct statement of \cite[Lemma~4.9]{KMN1} using the definition of $\Sigma$ via right actions there is  $\Sigma^1(F_{n,[-\infty,t_0]})^c=\{ [\nu_1], - [\nu_2]\}$, slightly  different than what is stated in \cite{KMN1} Lemma 4.9 as $\{[\nu_1],[\nu_2]\}$.
However, this does not affect the main results of \cite{KMN1}, since with the notation of  \cite{KMN1} Lemma 4.11, we get that the map
$\widetilde{\varphi}^*$ swaps
$[\tilde{\mu}_1]$ with $[\tilde{\nu}_1]$ and $[\tilde{\mu}_2]$ with $[-\tilde{\nu}_2]$ and in \cite{KMN1} the fact that $\widetilde{\varphi}^*$ swaps
$[\tilde{\mu}_1]$ with $[\tilde{\nu}_1]$ was used and this  holds in our corrected version.
\end{proof}

From now on, we consider the case $n=2$ only and put $F:=F_{2,\infty}$.
We consider the following two characters of $F$. Let $f\in F$ be such that  $(x)f=x+i$ when $x>>0$ and $(x)f=x+j$ for $x<<0$ and define
$$\chi_1(f)=j,$$
$$\chi_2(f)=-i.$$
Note that as $F /F'$ has rank 2, $S(F)=\{[a\chi_1+b\chi_2]\mid a,b\in\R\}$.

Let $I$ be the unit interval and $\phi:I\to\R$ be the PL-homeomorphism with breakpoints $1/2^i$, $1-1/2^j$ for $i,j> 0$ such that
$(1/2^i)\phi=-i+1$, $(1-1/2^j)\phi= j-1$. Conjugating with $\phi$ yields an isomorphism from $F_{2,\infty}$ to the Thompson group of PL-homeomorphisms of $I$ which we denote
 $F_I$. Using this isomorphism we get for $F_I$  the characters $\alpha_i(h)=\chi_i(\phi^{-1} h\phi)$ for $i=1,2$ and one easily checks that $\alpha_i(h):=\text{log}_2((i-1)h')$. Therefore by Theorem A in \cite{BGK} (see also \cite{BNS}) after sign change,
$$\Sigma^1(F)=S(F)\setminus\{- [\chi_1],- [\chi_2] \},$$
\begin{equation}\label{SigmaF}\Sigma^\infty(F,\Z)^c = S(F)\setminus \Sigma^2(F,\Z)=  conv_{\leq 2} \{- [\chi_1], - [\chi_2]\} = $$ $$
\{[a\chi_1+b\chi_2]\mid a \leq  0, b \leq 0, (a,b) \not= (0,0) \}.\end{equation}

\begin{proposition}\label{SextF} Let $G=K\ltimes F$ with $K$ finite such that $C_{F}(K)<F$. Then $S(G)=\{[\nu],[-\nu]\}$
with $\nu:G\to\R$ the only character such that $\nu |_F=\chi_1+\chi_2$.
\end{proposition}
\begin{proof} By the same argument of \cite{KMN1} Theorem C, there is a subgroup $K_0\leq K$ of index 2 acting trivially on $F$.
Then, $KF/K_0\cong (K/K_0) \ltimes F$, and the action of $K/K_0$ on $F$ is given by conjugation with a decreasing homeomorphism $h$ of $\R$ such that $h^2=id$ (see \cite{KMN1}, Lemma 4.1).
 Let $[\nu]\in S(G)$. By Lemma \ref{finext}, $\nu |_F\in Hom(F, \mathbb{R})^K=Hom(F, \mathbb{R})^{K/K_0}$. The fact that $h$ is decreasing implies that the induced action of $K/K_0$ in $S(F)$ swaps $[\chi_1]$ and $[\chi_2]$, this also follows taking into account that $\Sigma^1(F,\Z)^c=\{ - [\chi_1], - [\chi_2]\}$. Therefore
$$S(F)^{K/K_0}=\{[\chi_1+\chi_2], [-\chi_1-\chi_2]\}$$
and the claim follows by Lemma \ref{finext}.
\end{proof}

As a consequence, we get

\begin{lemma}\label{SigmaextF} Let $G=K\ltimes F$ with $K$ finite such that $C_{F}(K)<F$. Then $\Sigma^\infty(G,\Z)=\{[\nu]\}$ with $\nu |_F=\chi_1+\chi_2$.
\end{lemma}
\begin{proof}
It follows by \cite[Thm.~9.3]{MMV} (see Theorem \ref{finiteindex}), Proposition \ref{SextF} and (\ref{SigmaF}).
\end{proof}

\begin{theorem} \label{BredonThompson} Let $G=K\ltimes F$ with $K$ finite. Then
$$\underline{\Sigma}^\infty(G,\underline{\Z})=\Sigma^\infty(G,\Z).$$
\end{theorem}
\begin{proof} We only have to check that any $[\chi] \in\Sigma^\infty(G,\Z)$ also lies in $\underline{\Sigma}^\infty(G,\underline{\Z})$. As by Corollary D of \cite{KMN1} we know that $G$ has only finitely many conjugacy classes of finite subgroups, all we have to do is to check whether $\chi$ satisfies the conditions of  Theorem \ref{condition2}.

If $C_F(K)=F$, then $G=K\times F$ so both conditions from Theorem \ref{condition2} are trivial (taking $\widetilde G=G$).

So we may assume $C_F(K)<F$ thus by Lemma \ref{SigmaextF}, $[\nu]=[\chi]$ and we may assume $\chi = \nu$, in particular it is discrete (all the characters of $G$ are).
Fix
$$
\widetilde{G} = C_F(K) \text{Ker}(\nu).
$$
Observe that any finite subgroup of $G$ is contained in $\text{Ker}(\nu)$, hence is contained in $\widetilde{G}$.
This means that the result will follow if we prove that $\widetilde{G}$ has finite index in $G$ and for any finite $Q\leq G$,
$N_G(Q)\not\leq\text{Ker}(\nu)$ and $[\nu |_{N_G(Q)}] \in\Sigma^\infty(N_G(Q),\Z)$.   If $Q$ acts trivially on $F$, then last two assertions are obvious. And in other case,
arguing as in Proposition \ref{SextF}, we see that there is a $Q_0\leq Q$ of index 2 acting trivially on $F$ so that  the action of $Q/Q_0=\langle\varphi\rangle$ on $F$ is given by conjugation with a decreasing homeomorphism $h$ of $\R$ such that $h^2=id$. Note that $\chi_2^\varphi=\chi_1$.
Now, by  \cite[Thm.~4.14]{KMN1}  there is an isomorphism
$$\rho:C_{F}(\varphi)\to F_{2,[t_0,\infty]}$$
where $t_0\in\R$ is the only element such that $(t_0)h=t_0$ and  $\rho$ sends $f$ to its restriction on $[t_0,\infty]$. Furthermore it was shown in \cite[Thm.~7.3]{KMN1} that $t_0 \in \Z[{1 \over 2}]$, hence $F_{2,[t_0,\infty]} \simeq F$. By Lemma \ref{correction} $\Sigma^1(F_{2,[t_0,\infty]})^c=\{ -[\mu_1], -[\mu_2]\} $ and since $F_{2,[t_0,\infty]} \simeq F$  and by (\ref{SigmaF}) $$[\mu_2] \in \Sigma^{\infty}(F_{2,[t_0,\infty]}).$$
\noindent
Let $\mu:F_{2,[t_0,\infty]}\to\R$ be the character obtained by composing $\nu |_{C_F(\varphi)} \rho^{-1}$.  Since $C_F(\varphi)$ has finite index in $N_G(Q)$ and  by \cite[Theorem 9.3]{MMV} (see Theorem \ref{finiteindex}), we only have to check that $\mu\neq 0$ and $[\mu] \in\Sigma^{\infty}(F_{2,[t_0,\infty]})$. And to understand $\mu$ basically we only have to understand $\rho^{-1}$, which by \cite[Theorem 4.14]{KMN1}  sends $f\in F_{2,[t_0,\infty]}$ to the only $\widetilde f:\R\to\R$ with $(x)f=(x)\widetilde f$ for $x\in[t_0,\infty)$ and such that $\widetilde fh=h\widetilde f$. Assume that $(x)f=x+i$ for $i>>0$. Then as $\tilde f\in C_F(\varphi)$ and $\chi_2^\varphi=\chi_1$ we have
$$\mu(f)=\nu(\widetilde f)= \chi_1(\widetilde f) + \chi_2(\widetilde f)=\chi_2^\varphi(\widetilde f) + \chi_2(\widetilde f)= 2\chi_2(\widetilde f)= - 2i, $$
which means that $0\neq [\mu]= [\mu_2]\in\Sigma^\infty(F_{2,[t_0,\infty]},\Z)$.
In particular for $Q=K$ we get
$\text{Ker}(\nu)<C_F(K)\text{Ker}(\nu)=\widetilde G$ thus $[G : \widetilde{G}] < \infty$.
\end{proof}

\begin{theorem}  \label{finalThompson} Let $G$ be a finite extension of the R. Thompson group $F$. Then
$$\underline{\Sigma}^\infty(G,\underline{\Z})=\Sigma^\infty(G,\Z).$$
\end{theorem}

\begin{proof}
By Corollary D of \cite{KMN1} there are finitely many conjugacy classes of finite subgroups in $G$, let $K_1, \ldots, K_s$ be representatives of these conjugacy classes.
For $1 \leq i \leq s$ set $G_i = F_i \rtimes K_i$.

Let $[\chi] \in \Sigma^{\infty}(G, \mathbb{Z})$. We have to show that $[\chi] \in \underline{\Sigma}^\infty(G,\underline{\Z})$.
Note that by \cite[Thm.~9.3]{MMV} (see Theorem \ref{finiteindex})  for $i \leq s$ we have $[\chi \mid_{G_i}] \in \Sigma^{\infty}(G_i, \mathbb{Z})$. By the proof of Theorem \ref{BredonThompson} \begin{equation} \label{generalthompson1} [\chi \mid_{N_{G_i}(K_i)}] \in \Sigma^{\infty}(N_{G_i}(K_i), \mathbb{Z}). \end{equation}

Let $K$ be a finite subgroup of $G$, so $K = K_i^g$  for some $g,i$. Thus to establish the second condition of Theorem 6.6 it suffices to consider the case when $K = K_i$. Note that $[N_G(K_i) : N_{G_i}(K_i)] < \infty$. By (\ref{generalthompson1}) and  Theorem
\ref{finiteindex} we get that
$$
\chi(N_G(K_i)) \not= 0 \hbox{ and } [\chi \mid_{N_{G}(K_i)}] \in \Sigma^{\infty}(N_{G}(K_i), \mathbb{Z}),
$$
so the second condition of Theorem \ref{condition2} holds.

Set $$\widetilde{G} = \cap_{1 \leq i \leq s} (C_F(K_i) Ker (\chi)).$$ Note that for $g \in G$ we have $C_F(K_i^g) Ker (\chi) = C_F(K_i)^g Ker (\chi) = C_F(K_i) Ker (\chi)$, so $\widetilde{G}$ is the intersection of $C_G(K) Ker (\chi)$ where $K$ runs through all finite subgroups of $G$. We claim that $C_F(K_i) (Ker (\chi) \cap G_i) $ has finite index in $G_i$,  hence in $G$. Indeed if $K_i$ acts non-trivially on $F$  this follows from the proof of Theorem \ref{BredonThompson} and the fact that $[\chi \mid_{G_i}] \in \Sigma^{\infty}(G_i, \mathbb{Z})$, so $\chi \mid_{G_i}$ is the unique character of $G_i$ whose restriction on $F$ is $\chi_1 + \chi_2$. If $K_i$ acts trivially on $F$ we have $F \leq C_G(K_i)$, so $[G_i : C_F(K_i) (Ker (\chi) \cap G_i) ] \leq [G_i : F] < \infty$. Thus $\widetilde{G}$ has finite index in $G$. By Lemma \ref{conjfinite1} applied for the normal subgroup
$N = Ker (\chi)$ of $G$  we get that $\widetilde{G}_{\chi}$ conjugates finite subgroups, hence by Lemma \ref{conjfinite} applied for the group $\widetilde{G}$ the first condition of Theorem \ref{condition2} holds.

\end{proof}

\section{Soluble groups of type $FP_{\infty}$}

We recall some results that will be usefull in this section.

\begin{theorem} \cite[Cor.~5.2]{Meinert} \label{Meinert1} Let $G$ be a nilpotent-by-abelian group of type $FP_{\infty}$ then
$$
\Sigma^{\infty}(G, \mathbb{Z})^c = conv \Sigma^1(G, \mathbb{Z})^c
$$
\end{theorem}

 \begin{theorem}\cite[Thm.~9.3]{MMV} \label{finiteindex} Let $H$ be a subgroup of finite index in a finitely generated group $G$ and $\chi : G \to \mathbb{R}$ be a non-trivial character. Then $[\chi] \in \Sigma^m(G, \mathbb{Z})$ if and only if $[\chi |_{H}] \in \Sigma^m(H, \mathbb{Z})$.
 \end{theorem}

 \begin{theorem} \cite[Thm.~2.4]{ConcBri} \label{conjugacy} Let $\Gamma$ be a virtually soluble group of type $FP_{\infty}$. Then there is only a finite number of conjugacy classes of finite subgroups of $\Gamma$.
 \end{theorem}

 \begin{theorem} \cite[Thm.~3.13]{ConcBri} \label{soluble} Let $G$ be a virtually
soluble group of type $FP_{\infty}$ and $F$ be a finite group acting on $G$. Then $C_G(F)$ is of type $FP_{\infty}$ and is finitely presented.
 \end{theorem}

 We outline the main steps in the proof of Theorem \ref{soluble}.
 Recall for a $\Z Q$-module $A$ the invariants $\Sigma_A(Q) = \{ [\chi] \in S(Q) | A $ is finitely generated as $\Z Q_{\chi}$-module $\}$ and $\Sigma_A(Q)^c = S(Q) \setminus \Sigma_A(Q)$. For a subset $M$ of $S(Q)$ we denote by $dis M$ the discrete points of $M$.

 1. By the classification of soluble groups of type $FP_{\infty}$ started in \cite{GS}, and finished in \cite{Krop1},
such groups are  virtually torsion-free, constructible. Hence they are finitely presented and nilpotent-by-abelian-by-finite. So it suffices to assume that $G$ is nilpotent-by-abelian.

 2. For nilpotent-by-abelian groups $G$ it is known that $G$ is of type $FP_{\infty}$ if and only  $\Sigma^1(G, \mathbb{Z})^c$ lies in an open hemisphere in $S(G)$ \cite{BS}.

 3. By going down to a subgroup of finite index if necessary we can assume that  $G$ is nilpotent-by-abelian of type $FP_{\infty}$, with normal nilpotent subgroup $N$ and abelian quotient $Q = G /N$ such that $N$ and $Q$ are $F$-invariant. By going down to a subgroup of finite index again we can further assume that $Q$ is torsion-free and $Q = C_0 \times T_0$, where $F$ acts trivially on $C_0$ and $e = \sum_{t \in F} t$ acts as zero on $T_0$.

 4. Let $A$ be the abelianization of $N$, so the action of $F$ on $N$ induces an action of $Q$ on $A$. Under the assumptions of step 3 since $dis \Sigma_A^c(Q)$ is contained in some open hemisphere of $S(Q)$ then $A$ is finitely generated as $\mathbb{Z} C_0$-module (via the conjugation action of $C_0$) and $dis \Sigma_A^c(C_0)$ is contained in some open hemisphere of $S(C_0)$.

 5. Let $C$ be a subgroup of $G$ containing $N$ such that $C_{G/N}(F) = C / N$. Then $C$ is of type $FP_{\infty}$.

 6. The group $N C_G(F)$ has finite index in $C$. In particular $N C_G(F)$ is of type $FP_{\infty}$.

 7.  Let $S$ be a subgroup of $G$ such that $S N = G$. Then $S$ is of type $FP_{\infty}$.
 Applying this for $N C_G(F)$ in the place of $G$ we deduce that $C_G(F)$ is of type $FP_{\infty}$.

	\begin{lemma} \label{char1} Let $F$ be a finite group acting on a nilpotent-by-abelian group $G$ of type $FP_{\infty}$ as described in Step 3 above i.e. $N$ is nilpotent, $Q = G / N$  is torsion-free abelian, $N$ and hence $Q$ are $F$-invariant  and $Q = C_0 \times T_0$, where $F$ acts trivially on $C_0$ and $e = \sum_{t \in F} t$ acts as zero on $T_0$. Let $\Gamma$ be a finite index extension of $G$ that contains $F$ and $\chi : \Gamma \to \mathbb{R}$ be a non-trivial homomorphism such that $\chi(N) = 0$. Let $\widetilde{\chi} : Q \to \mathbb{R}$ be the homomorphism induced by $\chi$. Suppose that
	$$
	[\widetilde{\chi}] \notin conv \Sigma^c_A(Q),
	$$
where $A $ is the abelianization of $N$.
	Then
	$$
	[\widetilde{\chi} |_{C_0}] \notin conv \Sigma^c_A(C_0).
	$$
	\end{lemma}
	
	\begin{proof}
	By construction for $V = Q \otimes_{\mathbb{Z}} \mathbb{Q}$ and the idempotent  $\tilde{e} =  ({1 \over | F |}) \sum_{t \in F} t$
we have
$$
V = V (1 - \tilde{e}) \oplus  V \tilde{e},
$$
where  $T_0 \otimes_{\Z} \mathbb{Q} =   V (1 - \tilde{e}) $ and $C_0 \otimes_{\Z} \mathbb{Q} =  V\tilde{e}
$. Then $(T_0 \otimes_{\Z} \mathbb{Q}) \otimes_{\mathbb{Q} F} \mathbb{Q} = 0$, so the image of $T_0$ in the abelianization of $\Gamma$ is finite, in particular
\begin{equation} \label{zerovalue}
\tilde\chi(T_0) = 0.
\end{equation}
The rest of the proof is similar to  the proof of \cite[Lemma~3.6]{ConcBri}. We outline the main steps. First by the last paragraph of the proof of \cite[Prop.~3.9]{ConcBri} (there $C_0$ was denoted by $C$) $A$ is finitely generated as $\mathbb{Z} C_0$-module.
Suppose that
$$
\widetilde{\chi} |_{C_0} = \chi_1 + \ldots + \chi_m,
$$
where $[\chi_i] \in \Sigma^c_A(C_0)$. By the link between $\Sigma^c$ and valuations \cite[Thm.~8.1]{BG} there is $[v_i] \in \Sigma^c_A(Q)$ such that the restriction of $v_i$ to $C_0$ is $\chi_i$ i.e. $v_i = (\chi_i, w_i)$, where $w_i$ is the restriction of $v_i$ on $T_0$. Then
$$
\sum_{t \in F} v_i^t = \sum_{t \in F} (\chi_i, w_i^t) = (|F| \chi_i,  \sum_{t \in F} w_i^t) = (|F| \chi_i, 0).
$$
Hence
$$
|F| \widetilde{\chi} = |F| (\widetilde{\chi} |_{C_0}, \widetilde{\chi} |_{T_0} =  0) = \sum_{t \in F, 1 \leq i \leq m} v_i^t
\hbox{ , thus }
[\widetilde{\chi}] \in conv \Sigma_A^c(Q),$$
a contradiction.

\end{proof}

\begin{lemma} Let $F$ be a finite group acting on a group $G$. Assume that $G$ has a normal $F$-invariant nilpotent subgroup $N$, $Q = G / N$  is torsion-free abelian and $Q = C_0 \times T_0$, where $F$ acts trivially on $C_0$ and $e = \sum_{t \in F} t$ acts as zero on $T_0$.  Let $\Gamma$ be a finite extension of $G$ that contains $F$ and $\chi : \Gamma \to \mathbb{R}$ be a non-trivial homomorphism such that $$\chi(N) = 0 \hbox{ and }[\chi] \in \Sigma^{\infty}(\Gamma, \mathbb{Z}).$$
Let  $C$ be the subgroup of $G$ containing $N$ and such that $C / N = C_{G/N} (F)$. Then
$$
[\chi |_C] \in \Sigma^{\infty}(C, \mathbb{Z}).
$$
\end{lemma}

\begin{proof}
Observe that $\Sigma^{\infty}(\Gamma, \mathbb{Z}) \not= \emptyset$ implies that $\Gamma$ and hence $G$ are of type $FP_{\infty}$.
By Theorem \ref{Meinert1} and Theorem \ref{finiteindex}
$$
[\chi |_{G}] \in \Sigma^{\infty}(G, \mathbb{Z}) = S(G) \setminus (conv \Sigma^1(G, \mathbb{Z})^c).
$$
Since $N$ is nilpotent by \cite[Thm.~2.3]{Meinert}
\begin{equation} \label{==}
\Sigma^1(G, \mathbb{Z})^c = \Sigma_{A}^c(Q),
\end{equation}
where $A $ is the abelianization of $N$ i.e. for a real homomorphism $\mu : G \to \mathbb{R}$  we have $[\mu] \in \Sigma^1(G,\Z)^c$ if and only if $\mu(N) = 0$ and for the homomorphism $\widetilde{\mu} : Q \to \mathbb{R}$ induced by $\mu$ we have $[\widetilde{\mu}] \in \Sigma_A^c(Q) = S(Q) \setminus \Sigma_A(Q)$. Hence for the character $\widetilde{\chi} : Q \to \mathbb{R}$ induced by $\chi$ we have
$$
[\widetilde{\chi}] \notin conv \Sigma_{A}^c(Q).
$$
Then by Lemma \ref{char1}
\begin{equation} \label{novo99}
	[\widetilde{\chi} |_{C_0}] \notin conv \Sigma^c_A(C_0).
	\end{equation}
	Since $C_0 = C / N$, as in (\ref{==}) we have $\Sigma_A^c(C_0) = \Sigma^1(C, \Z)^c$ and by  Theorem \ref{Meinert1} and (\ref{novo99})  we obtain
	$$
	[\chi |_C] \notin conv \Sigma^1(C, \mathbb{Z})^c = \Sigma^{\infty}(C, \mathbb{Z})^c, \hbox{ so }  [\chi |_C] \in \Sigma^{\infty}(C, \mathbb{Z}).
	$$
	\end{proof}

\begin{lemma} Let $G$ be a group with a normal nilpotent subgroup $N$ and $Q = G/ N$ abelian. Let $\chi : G \to \mathbb{R}$ be a non-trivial character such that  $[\chi] \in \Sigma^{\infty}(G, \mathbb{Z})$ and $\chi(N) = 0$. Let $S$ be a subgroup of $G$ such that $S N = G$. Then
$$
[\chi |_S] \in \Sigma^{\infty}(S, \mathbb{Z}).
$$
\end{lemma}

\begin{proof} Since $\Sigma^{\infty}(G, \mathbb{Z}) \not= \emptyset$, $G$ is of type $FP_{\infty}$. By \cite[Lemma~3.12]{ConcBri} $S$ is of type $FP_{\infty}$ and
by Theorem \ref{Meinert1}
$$
\Sigma^{\infty}(S, \mathbb{Z})^c = conv \Sigma^1(S, \mathbb{Z})^c.
$$
Recall that as in  (\ref{==})
$$
\Sigma^1(S, \mathbb{Z})^c = \Sigma_B^c(Q),
$$
where $B = S \cap N / [S \cap N, S \cap N]$ and $Q = G / N \simeq S / S \cap N$.
Hence to prove the lemma we have to show for the character $\widetilde{\chi} : Q \to \mathbb{R}$ induced by $\chi |_{S}$ that
\begin{equation} \label{sigmaB} [\widetilde{\chi}] \notin conv \Sigma_B^c(Q).
\end{equation}
Note that $S$ and $G$ are of type $FP_{\infty}$, so $\Sigma_B^c(Q)$ and $\Sigma_A^c(Q)$ contain only discrete points, where $A = N / [N,N]$. As in the proof of \cite[Lemma~3.12]{ConcBri}
$$\Sigma_B^c(Q) = dis \Sigma_B^c(Q) \subseteq conv \ \ dis \Sigma_A^c(Q) = conv \Sigma_A^c(Q),$$
hence
\begin{equation} \label{inclusion}
conv \Sigma_B^c(Q) \subseteq conv \Sigma_A^c(Q).\end{equation}
Using again Theorem \ref{Meinert1}
$$
[\chi] \notin \Sigma^{\infty}(G, \mathbb{Z})^c = conv \Sigma^1(G, \mathbb{Z})^c
$$
and as in (\ref{==})
$$
\Sigma^1(G, \mathbb{Z})^c = \Sigma_A^c(Q).
$$
Since $\chi(N) = 0$, for $\widetilde{\chi}$ the character induced by $\chi$,  we have
$$
[\widetilde{\chi}] \notin conv \Sigma_A^c(Q).
$$
Then by (\ref{inclusion}) we deduce that (\ref{sigmaB}) holds.
\end{proof}

  We finish the section by  proving the following
  $\Sigma$-version of Theorem \ref{soluble}.

  \begin{theorem} \label{Bredonsoluble} Let $\Gamma$ be a virtually soluble group of type $FP_{\infty}$. Then
  $$\underline{\Sigma}^{\infty} (\Gamma, \underline{\mathbb{Z}}) = {\Sigma}^{\infty} (\Gamma, \mathbb{Z}).$$
  \end{theorem}

	\begin{proof} Note that  by Theorem \ref{conjugacy} $\Gamma$ has finitely many conjugacy classes of finite subgroups. Observe that by
	Remark \ref{subset} we have
	$
	\underline{\Sigma}^{m} (\Gamma, \underline{\mathbb{Z}}) \subseteq {\Sigma}^{m} (\Gamma, \mathbb{Z}).
	$

For the converse let $\chi : \Gamma \to \mathbb{R}$  be a non-zero homomorphism such that
$[\chi] \in {\Sigma}^{\infty} (\Gamma, \mathbb{Z})$
and $K$ be a finite subgroup of $\Gamma$. Let $G$ be a normal nilpotent-by-abelian subgroup of $\Gamma$. Then $K$ acts on $G$ via conjugation and $C_G(K)$ has finite index in $C_{\Gamma}(K)$.
By substituting $G$ with  a subgroup of finite index if necessary we can assume that the assumptions of Step 3 hold and $\chi(N) = 0$.
Then the previous two  lemmas imply that $
[\chi |_S] \in \Sigma^{\infty}(S, \mathbb{Z})$ for $S = C_G(K)$. Since $S$ has finite index in $D = C_{\Gamma}(K)$ by Theorem \ref{finiteindex} $[\chi |_D] \in \Sigma^{\infty}(D, \mathbb{Z})$.

By the line above (\ref{zerovalue}) and the fact that $C_G(F) N$ has finite index in the preimage of $C_0$ in $G$ we deduce that $C_{\Gamma}(F) [\Gamma, \Gamma]$ has finite index in $\Gamma$, hence we can define $$\widetilde G:=\cap\{C_{\Gamma}(F) [\Gamma, \Gamma] \mid F\text{ rep. of the conjugacy classes
of finite subgroups in } \Gamma\}.$$
Then by Lemma \ref{conjfinite} and Lemma \ref{conjfinite1} the first condition of Theorem \ref{condition2} holds.

Finally the proof is completed by Theorem \ref{condition2}, where condition 2 is substituted with condition 2b).
\end{proof}

\end{document}